\documentclass[11pt]{article}

\usepackage[margin=1.15in]{geometry}
\usepackage{amsmath,amssymb,amsthm,mathtools}
\usepackage{enumitem}
\usepackage[hidelinks]{hyperref}
\usepackage{microtype}

\newtheorem{theorem}{Theorem}[section]
\newtheorem{lemma}[theorem]{Lemma}
\newtheorem{corollary}[theorem]{Corollary}
\theoremstyle{remark}
\newtheorem{remark}[theorem]{Remark}

\usepackage{xcolor}

\newcommand{\C}{\mathbb C}
\newcommand{\D}{\mathbb D}
\newcommand{\T}{\mathbb T}
\newcommand{\B}{\mathcal B}

\title{Dilation Theory for Scaled $q$-Numerical Ranges}
\author{
Matias Kaunisto\thanks{
Department of Physics and Mathematics, University of Eastern Finland,
P.O. Box 111, FI-80101 Joensuu, Finland.
Email: \texttt{matias.kaunisto@gmail.com}
}
\and
Ryan O'Loughlin\thanks{
Department of Mathematics and Statistics, University of Reading,
Reading, England.
Email: \texttt{r.d.oloughlin@reading.ac.uk}
}
\and
Jani Virtanen\thanks{
Department of Physics and Mathematics, University of Eastern Finland,
80100 Joensuu, Finland;
Department of Mathematics and Statistics, University of Reading,
Whiteknights, Reading RG6 6AX, United Kingdom;
Department of Mathematics and Statistics, University of Helsinki,
P.O. Box 68, FI-00014 University of Helsinki, Finland.
Email: \texttt{jani.virtanen@uef.fi}
}
}
\date{}

\begin{document}
\maketitle

\begin{abstract}
We generalise Berger's dilation theorem from numerical ranges to scaled $q$-numerical ranges. Specifically, we prove that if the scaled $q$-numerical range $\frac{W_q(A)}{q}$ is contained in the closed unit disk, then $A$ admits a unitary $\rho_q$-dilation, where
\[
\rho_q=\max\left\{1,\frac{2|q|}{1+\sqrt{1-|q|^2}}\right\}.
\]
As a corollary we show that every disk containing $\Omega_q(A)$ is a complete $\rho_q$-spectral set for $A$. At $q=1$, this recovers the classical fact that any disk containing $W(A)$ is a complete $2$-spectral set for $A$.
\end{abstract}



\section{Introduction}\label{sec:introduction}

Let $A$ be a bounded operator on a complex Hilbert space $H$ with
$\dim H\geq2$, and let $q\in\C$ with $0<|q|\leq1$. The $q$-numerical range and scaled $q$-numerical range are, respectively,
\[
W_q(A)=\{\langle Ax,y\rangle:\|x\|=\|y\|=1,\ \langle x,y\rangle=q\},
\qquad
\Omega_q(A)=\frac{W_q(A)}{q}.
\]
The $q$-numerical range and its basic properties have been studied extensively (see, for example, \cite{li1994generalized,tsing1984constrained,gau2021numerical}). When
$q=1$, we recover the usual numerical range $W(A)$.

A classical theorem of Berger \cite{berger1965strange} states that
\[
W(A)\subseteq\overline{\mathbb D}
\]
if and only if $A$ admits a unitary $2$-dilation. Equivalently, operators
of numerical radius at most one are precisely the operators of class
$C_2$. Despite the classical importance of this result, comparatively
few direct geometric generalisations of Berger's theorem have appeared.
A recent exception is the work of Agler, Lykova and Young
\cite{agler2024ellipse}, who obtained an analogous dilation
characterisation for operators whose numerical range is contained in an
ellipse. Such dilation results are closely connected with functional
calculus estimates and, in particular, with the theory of spectral sets.

If a compact set $\Omega\subseteq\C$ contains the spectrum of $A$ and
$\|f(A)\|\leq K\sup_{z\in\Omega}|f(z)|$ for every rational function $f$ with poles off $\Omega$, then $\Omega$ is called a $K$-spectral set for $A$. It is a complete $K$-spectral set if, for every $m\geq1$ and every $M_m(\C)$-valued rational function $F=(f_{ij})$ with poles off $\Omega$,
\[
\|F(A)\|\leq K\sup_{z\in\Omega}\|F(z)\|,
\]
where $F(A)=(f_{ij}(A))$ acts on $H^m$. Clearly, taking $m=1$ shows every complete spectral set is a spectral set.

Crouzeix conjectured in \cite{CC2matrix} that the numerical range is a $2$-spectral set. He subsequently proved the numerical range is a complete $11.08$-spectral set in \cite{CC11.08}, and Crouzeix and Palencia reduced this to $1+\sqrt2$ in \cite{1root2}. The (non-complete) Crouzeix conjecture was recently proved by Jin \cite{jin2026} and by Lorist and Schwenninger \cite{loristschwenninger2026}. Crouzeix's conjecture and related problems have attracted considerable
further attention in functional and complex analysis communities \cite{o2024crouzeix,CC11.08,CC2matrix,CC3by3,CCgorkin,CCincreasing,CCLi,CCpalenciaextension,CCnilpotent,CCnum,malman2024double,1root2,mashreghi2025schwarz}. Recently, the complete Crouzeix conjecture was proved for matrices of order at most three in \cite{ahagczyzperalavirtanen2026}, but the general complete Crouzeix conjecture remains open.

A scaled $q$-numerical range analogue of Crouzeix's conjecture was proposed in
\cite{oloughlinrani2026}, with the coefficient
\begin{equation}\label{eq:rhoq}
\rho_q=\max\left\{1,\frac{2|q|}{1+\sqrt{1-|q|^2}}\right\}.
\end{equation}
More precisely, it was conjectured that for every matrix $A$, every
$0<|q|\leq1$, and every polynomial $p$,
\begin{equation}\label{conj}
\|p(A)\|
\leq
\rho_q\sup_{z\in\Omega_q(A)}|p(z)|.
\end{equation}
Since $\Omega_q(A)$ is convex, Mergelyan's theorem shows that \eqref{conj} holding for polynomials is equivalent to requiring it for rational functions with poles off $\overline{\Omega_q(A)}$, and hence \eqref{conj} is equivalent to $\Omega_q(A)$ being a $\rho_q$-spectral set for $A$. The conjecture \eqref{conj} was subsequently proved independently in
\cite{aouichaouioloughlin2026} and \cite{jin2026}. The recent work \cite{ahagczyzperalavirtanen2026} proves the corresponding sharp complete result for matrices of order two and three and shows, in every finite dimension, that $\Omega_q(A)$ is a complete spectral set whenever $|q|\leq\frac{1}{\sqrt2}$.


The main result of the present paper is a generalisation of Berger's 1965 numerical range dilation theorem \cite{berger1965strange} to scaled $q$-numerical ranges. In contrast to \cite{agler2024ellipse} our result instead retains the containing set as a disk and replaces the numerical range by the scaled $q$-numerical range. As a secondary consequence of this dilation result, we settle the circular case of the complete scaled $q$-numerical range conjecture for bounded operators on arbitrary Hilbert spaces.

\section{The dilation result}\label{sec:dilation}

We write $\D=\{z\in\C:|z|<1\}$, $\overline{\D}$ for its closure and $\T=\partial\D$. For $\rho\geq1$, a bounded operator $A$ on a Hilbert space $H$ belongs to the class $C_\rho$ if there is a Hilbert space $\mathcal{K}$ with $H \subseteq   \mathcal K$ and a unitary operator $U\in\B(\mathcal K)$ such that
\begin{equation}\label{eq:rhodilation}
A^m=\rho P_HU^m|_H,\qquad m=1,2,\ldots,
\end{equation}
where $P_H$ is the orthogonal projection of $\mathcal K$ onto $H$; see \cite{sznagyfoias1966}. In this case, $U$ is called a unitary $\rho$-dilation of $A$. For $1\leq\rho<2$, we will use the equivalent characterization
\begin{equation}\label{eq:rhocharacterization}
A\in C_\rho
\quad\Longleftrightarrow\quad
\|(2-\rho)A+(\rho-1)\zeta I\|\leq1
\quad(\zeta\in\T),
\end{equation}
which follows from the standard positivity criterion for the classes $C_\rho$ \cite{sznagyfoias1966}.

\begin{lemma}\label{lem:coeff}
Let $0<q<1$ and suppose that $\Omega_q(A)\subseteq\overline\D$. For every pair of unit vectors $x,y\in H$,
\begin{equation}\label{eq:coeffbound}
\begin{aligned}
|\langle Ax,y\rangle|
&\leq
\frac{q}{1+\sqrt{1-q^2}}
\left(1+\sqrt{1-|\langle x,y\rangle|^2}\right)
&&\text{if }|\langle x,y\rangle|\leq q,\\
|\langle Ax,y\rangle|
&\leq |\langle x,y\rangle|
&&\text{if }|\langle x,y\rangle|\geq q.
\end{aligned}
\end{equation}
\end{lemma}

\begin{proof}
Multiplying $y$ by a unimodular scalar, we may assume that $\langle x,y\rangle\geq0$. This does not change either modulus in \eqref{eq:coeffbound}. If $\langle x,y\rangle\geq q$, then
\[
\frac{\langle Ax,y\rangle}{\langle x,y\rangle}
\in \Omega_{\langle x,y\rangle}(A)
\subseteq \Omega_q(A)
\subseteq \overline{\mathbb D},
\]
where the inclusion $\Omega_t(A)\subseteq\Omega_r(A)$ for $0<r\leq t\leq1$ follows from the proof of \cite[Theorem~2.5]{li1994generalized}, which extends straightforwardly from matrices to bounded operators. Hence
$
|\langle Ax,y\rangle|
\leq |\langle x,y\rangle|,
$
which proves the second estimate in \eqref{eq:coeffbound}. 

Suppose that $0\leq\langle x,y\rangle<q$. Set
\[
e=\frac{x+y}{\sqrt{2(1+\langle x,y\rangle)}},
\qquad
f=\frac{x-y}{\sqrt{2(1-\langle x,y\rangle)}},
\qquad
b=\frac{\langle x,y\rangle}{1+\sqrt{1-\langle x,y\rangle^2}}.
\]
Then $e$ and $f$ are orthonormal. For $z\in\C$, put
$$
X_z=(1+z)e+(1-z)f, \quad Y_z=(1+\overline z)e-(1-\overline z)f,$$
and define $F(z)=\frac{\langle AX_z,Y_z\rangle}{4}$. A direct computation gives
\begin{equation}\label{eq:XzYz}
\|X_z\|^2=\|Y_z\|^2=2(1+|z|^2),
\qquad
\langle X_z,Y_z\rangle=4z.
\end{equation}

If
$
|z|=\frac{q}{1+\sqrt{1-q^2}},
$
then using \eqref{eq:XzYz}, a computation gives
$$
\frac{|\langle X_z,Y_z\rangle|}
{\|X_z\|\|Y_z\|}
=q.
$$
Thus, since by assumption
$
W_q(A)\subseteq q\overline{\mathbb D},
$
we have
\[
\frac{|\langle AX_z,Y_z\rangle|}
{\|X_z\|\|Y_z\|}
\leq q.
\]
Hence, when
$
|z|=\frac{q}{1+\sqrt{1-q^2}},
$
\[
|F(z)|
\leq
\frac{q}{4}\|X_z\|\|Y_z\|.
\]
In this case a further direct computation from \eqref{eq:XzYz} yields $\frac{q}{4}\|X_z\|\|Y_z\|
=\frac{q}{1+\sqrt{1-q^2}}$. Noting $\frac{s}{1+\sqrt{1-s^2}}$ is increasing on $[0,1)$, shows $
b\leq \frac{q}{1+\sqrt{1-q^2}}
$
and clearly $F(z)$ is a polynomial, so applying the maximum-modulus principle on the disk of radius $\frac{q}{1+\sqrt{1-q^2}}$ yields \begin{equation}\label{maxmod}
    |F(b)|\leq \frac{q}{1+\sqrt{1-q^2}}.
\end{equation}
A straightforward computation from the definitions gives
\[
X_b=\sqrt{2(1+b^2)}\,x,
\qquad
Y_b=\sqrt{2(1+b^2)}\,y,
\]
and thus
\[
F(b)=\frac{1+b^2}{2}\langle Ax,y\rangle.
\]
Since
\[
\frac{2}{1+b^2}
=
1+\sqrt{1-|\langle x,y\rangle|^2},
\]
it follows from \eqref{maxmod} that
\[
|\langle Ax,y\rangle|
=
\left(1+\sqrt{1-|\langle x,y\rangle|^2}\right)|F(b)|
\leq
\frac{q}{1+\sqrt{1-q^2}}
\left(1+\sqrt{1-|\langle x,y\rangle|^2}\right).
\]
\end{proof}

The following theorem is the main dilation result of the paper. Berger's theorem corresponds to the endpoint $q=1$; the result below shows that this classical numerical range theorem forms part of a broader family of dilation results for scaled $q$-numerical ranges.

\begin{theorem}\label{thm:dilation}
Let $A\in\B(H)$ and let $0<|q|\leq1$. If $\Omega_q(A)\subseteq\overline\D$, then $A\in C_{\rho_q}$, where $\rho_q$ is given by \eqref{eq:rhoq}.
\end{theorem}

\begin{proof}
Since $\Omega_q(A)=\Omega_{|q|}(A)$, we may assume that $q>0$.
If $q=1$, then $\Omega_1(A)=W(A)$ and the result is the classical
disk theorem of Berger \cite{berger1965strange}. So we assume throughout that $0<q<1$, which means $1\leq\rho_q<2$, and consequently by \eqref{eq:rhocharacterization} it is enough to prove
\begin{equation}\label{eq:targetnorm}
\|(2-\rho_q)A+(\rho_q-1)\zeta I\|\leq1,
\qquad \zeta\in\T.
\end{equation}
Fix $\zeta\in\T$ and unit vectors $x,y\in H$. The triangle inequality gives
\begin{equation}\label{eq:triangle}
\left|\left\langle\big((2-\rho_q)A+(\rho_q-1)\zeta I\big)x,y\right\rangle\right|
\leq
(2-\rho_q)|\langle Ax,y\rangle|+(\rho_q-1)|\langle x,y\rangle|.
\end{equation}
If $|\langle x,y\rangle|\geq q$, Lemma~\ref{lem:coeff} then gives
\[
\begin{aligned}
(2-\rho_q)|\langle Ax,y\rangle|
+(\rho_q-1)|\langle x,y\rangle|
&\leq
(2-\rho_q)|\langle x,y\rangle|
+(\rho_q-1)|\langle x,y\rangle|\\
&=
|\langle x,y\rangle|\\
&\leq 1.
\end{aligned}
\]

Suppose that $|\langle x,y\rangle|\leq q$. Since
$$
\frac{q}{1+\sqrt{1-q^2}}\leq\frac{\rho_q}{2},
$$
Lemma~\ref{lem:coeff} bounds the right-hand side of \eqref{eq:triangle} by
\begin{equation}\label{EQCS}
    \frac{\rho_q(2-\rho_q)}{2}
    \left(1+\sqrt{1-|\langle x,y\rangle|^2}\right)
    +(\rho_q-1)|\langle x,y\rangle|.
\end{equation}
Applying Cauchy--Schwarz to the vectors
\[
\left(
\frac{\rho_q(2-\rho_q)}{2},
\,\rho_q-1
\right)
\quad\text{and}\quad
\left(
\sqrt{1-|\langle x,y\rangle|^2},
\,|\langle x,y\rangle|
\right)
\]
shows the expression in \eqref{EQCS} is bounded above by
\[
\frac{\rho_q(2-\rho_q)}{2}
+
\sqrt{
\frac{\rho_q^2(2-\rho_q)^2}{4}
+(\rho_q-1)^2
}.
\]
Then noting $ \rho_q(2-\rho_q)=1-(\rho_q-1)^2$, we have
\[
\sqrt{
\frac{\rho_q^2(2-\rho_q)^2}{4}
+(\rho_q-1)^2
}
=
\frac{1+(\rho_q-1)^2}{2}.
\]
Therefore
\[
\frac{\rho_q(2-\rho_q)}{2}
+
\sqrt{
\frac{\rho_q^2(2-\rho_q)^2}{4}
+(\rho_q-1)^2
}
=
\frac{1-(\rho_q-1)^2}{2}
+
\frac{1+(\rho_q-1)^2}{2}
=1.
\]
Hence the expression in \eqref{EQCS} is at most $1$ and therefore
\[
\left|
\left\langle
\big((2-\rho_q)A+(\rho_q-1)\zeta I\big)x,y
\right\rangle
\right|
\leq 1
\]
for all unit vectors $x,y\in H$ so taking the supremum over unit $x$ and $y$ shows \eqref{eq:targetnorm} holds.
\end{proof}

\begin{remark}\label{rem:affine}
The relation $\Omega_q(\alpha A+\beta I)=\alpha\Omega_q(A)+\beta$ shows immediately that if $\Omega_q(A)\subseteq c+R\overline\D$, with $R>0$, then $\frac{A-cI}{R}\in C_{\rho_q}$.
\end{remark}

\begin{remark}\label{rem:sharpdilation}
The coefficient in Theorem~\ref{thm:dilation} is sharp. For $\frac{4}{5}<q\leq1$, let
\[
A=\frac{2q}{1+\sqrt{1-q^2}}
\begin{pmatrix}0&1\\0&0\end{pmatrix}.
\]
A direct $2\times2$ calculation gives
$$
W_q\left(\begin{pmatrix}0&1\\0&0\end{pmatrix}\right)=\frac{1+\sqrt{1-q^2}}{2}\overline\D,$$
and hence $\Omega_q(A)=\overline\D$ and $\|A\|=\rho_q$. Any unitary $\rho$-dilation implies $\|A\|\leq\rho$ by taking $m=1$ in \eqref{eq:rhodilation}, so no smaller coefficient is possible. For $0<q\leq\frac{4}{5}$, we have $\rho_q=1$, which is already the least value allowed in the classes $C_\rho$.
\end{remark}

\section{Complete spectral sets}\label{sec:complete}

The dilation result has an immediate complete spectral-set consequence. In the classical case, when $q=1$, it is known that when the closure of the numerical range is a disk it is a complete $2$-spectral set \cite{okubo1975constants,davidsonpaulsenwoerdeman}. The following result extends this circular numerical range result to scaled $q$-numerical ranges, using a standard argument from the theory of $\rho$-contractions.

\begin{corollary}\label{cor:complete}
Let $A\in\B(H)$, let $0<|q|\leq1$ and let $\Omega_q(A)\subseteq\overline\D$. Then the closed unit disk is a complete $\rho_q$-spectral set for $A$. 
\end{corollary}

\begin{proof}
By Theorem~\ref{thm:dilation}, $A\in C_{\rho_q}$ and so by
\cite[Theorem~2]{okubo1975constants}, there exists a bounded invertible
operator $S$ such that
\[
\|S^{-1}AS\|\leq1
\qquad\text{and}\qquad
\|S\|\|S^{-1}\|\leq\rho_q.
\]
Set $T=S^{-1}AS$ and let $F$ be an $M_m(\mathbb C)$-valued rational function
with poles off $\overline{\mathbb D}$, so
\begin{equation}\label{eq:similarityF}
F(A)=(I_m\otimes S)F(T)(I_m\otimes S^{-1}).
\end{equation}
Since $\|T\|\leq1$, the complete von Neumann inequality
\cite{paulsen2002completely} gives
\[
\|F(T)\|\leq\sup_{|z|\leq1}\|F(z)\|.
\]
Combining this with \eqref{eq:similarityF} and using
$\|I_m\otimes S\|=\|S\|$ and
$\|I_m\otimes S^{-1}\|=\|S^{-1}\|$, we obtain
\begin{equation}\label{eq:completebound}
\|F(A)\|
\leq
\|S\|\|S^{-1}\|\,\|F(T)\|
\leq
\rho_q\sup_{|z|\leq1}\|F(z)\|.
\end{equation}
Thus $\overline{\mathbb D}$ is a complete $\rho_q$-spectral set for $A$.
\end{proof}

\begin{remark}\label{rem:circular}
If $\overline{\Omega_q(A)}=c+R\overline\D$, then Corollary~\ref{cor:complete} shows that the closure of the scaled $q$-numerical range is a complete $\rho_q$-spectral set for $A$. The coefficient is sharp: for $\frac{4}{5}<|q|\leq1$, the $2\times2$ operator in Remark~\ref{rem:sharpdilation}, together with the scalar polynomial $p(z)=z$, forces the constant $\rho_q$, while for $0<|q|\leq\frac{4}{5}$ the constant polynomial $p(z)=1$ shows that no constant below $1$ is possible.
\end{remark}

\begin{remark}
In finite dimensions, Corollary~\ref{cor:complete} can also be obtained from the similarity results in \cite[Theorem~2.2 and Corollary~2.3]{aouichaouioloughlin2026}, together with \cite[Proposition~9.1]{ahagczyzperalavirtanen2026}.
\end{remark}

\section*{Acknowledgements}

R. O'Loughlin and J. Virtanen were supported by EPSRC grant EP/Y008375/1.

\noindent Part of this work was carried out as part of M. Kaunisto's
Master's thesis at the University of Eastern Finland.

\bibliographystyle{plain}
\bibliography{bibliography}

\end{document}